\documentclass{amsart}
\usepackage{amssymb,amscd,amsmath,mathabx,enumerate,verbatim,calc}
\usepackage{amsopn,amsthm,graphics,amsfonts}
\usepackage{mathrsfs}
\usepackage[T1]{fontenc}
\usepackage[autostyle]{csquotes}
\usepackage[dvips]{graphicx}
\usepackage[autostyle]{csquotes}
\usepackage[colorlinks=true,linkcolor=red,citecolor=blue]{hyperref}
\usepackage{tikz-cd}
\usetikzlibrary{matrix,arrows,backgrounds}
\usepackage{tikz}
\usetikzlibrary{matrix,arrows,decorations.pathmorphing}
\input xy
\xyoption{all}
\calclayout

\makeatletter
\newcommand{\xRrightarrow}[2][]{\ext@arrow 0359\Rrightarrowfill@{#1}{#2}}
\newcommand{\Rrightarrowfill@}{\arrowfill@\equiv\equiv\Rrightarrow}
\newcommand{\xLleftarrow}[2][]{\ext@arrow 3095\Lleftarrowfill@{#1}{#2}}
\newcommand{\Lleftarrowfill@}{\arrowfill@\Lleftarrow\equiv\equiv}
\newcommand{\xLleftRrightarrow}[2][]{\ext@arrow 3399\LleftRrightarrowfill@{#1}{#2}}
\newcommand{\LleftRrightarrowfill@}{\arrowfill@\Lleftarrow\equiv\Rrightarrow}
\newcommand*{\medcap}{\mathbin{\scalebox{1.5}{\ensuremath{\cap}}}}%
\makeatother

\newcommand{\rt}{\rightarrow}

\newcommand{\st}{\stackrel}

\newcommand{\CY}{\mathscr{Y} }

\newcommand{\Prod}{{\rm{Prod}}}

\newcommand{\im}{{\rm{Im}}}

\newcommand{\Add}{{\rm{Add}}}
\newcommand{\add}{{\rm{add}}}

\newcommand{\gen}{{\rm{gen}}}
\newcommand{\cogen}{{\rm{cogen}}}
\newcommand{\Cogen}{{\rm{Cogen}}}

\newcommand{\RNum}[1]{\uppercase\expandafter{\romannumeral #1\relax}}

\newcommand{\pd}{{\rm{pd}}}
\newcommand{\id}{{\rm{id}}}
\newcommand{\fd}{{\rm{fd}}}

\newcommand{\Tr}{{\rm{Tr}}}

\newcommand{\Coker}{{\rm{Coker}}}
\newcommand{\Ker}{{\rm{Ker}}}

\newcommand{\Tor}{{\rm{Tor}}}
\newcommand{\Hom}{{\rm{Hom}}}
\newcommand{\Ext}{{\rm{Ext}}}
\newcommand{\End}{{\rm{End}}}

\newtheorem{theorem}{Theorem}[section]
\newtheorem{corollary}[theorem]{Corollary}
\newtheorem{lemma}[theorem]{Lemma}

\newtheorem{question}[theorem]{Question}

\newtheorem{definition}[theorem]{Definition}
\newtheorem{example}[theorem]{Example}

\newtheorem{remark}[theorem]{Remark}

\theoremstyle{plain}

\theoremstyle{definition}

\numberwithin{equation}{section}

\begin{document}

\author[K. Divaani-Aazar, A. Mahin Fallah and M. Tousi]
{Kamran Divaani-Aazar, Ali Mahin Fallah and Massoud Tousi}

\title[Comparing four definitions of ...]
{Comparing four definitions of cotilting modules}

\address{K. Divaani-Aazar, Department of Mathematics, Faculty of Mathematical Sciences,
Alzahra University, Tehran, Iran.}
\email{kdivaani@ipm.ir}

\address{A. Mahin Fallah, School of Mathematics, Institute for Research in Fundamental Sciences (IPM),
P.O. Box: 19395-5746, Tehran, Iran.}
\email{amfallah@ipm.ir, ali.mahinfallah@gmail.com}

\address{M. Tousi, Department of Mathematics, Faculty of Mathematical Sciences, Shahid Beheshti
University, Tehran, Iran.}
\email{mtousi@ipm.ir}


\keywords{Big cotilting module; $n$-tilting module; product-complete module; right Morita ring; Wakamatsu 
tilting module.\\
The research of the second author is supported by a grant from IPM (No.1403160017).}

\begin{abstract} In contrast to the theory of tilting modules, the dual theory lacks a unified definition.
Nevertheless, several notions of cotilting modules have been proposed. In this paper, we compare four of 
the main definitions of cotilting modules that have appeared in the literature. We show that, in the setting 
of finitely generated right modules, three of these definitions coincide over right Artinian Noetherian 
algebras, and all four coincide over Artin algebras.
\end{abstract}

\maketitle

\section{Introduction}

Tilting theory plays a fundamental role in the representation theory of algebras. 1-tilting modules were
initially defined over finite-dimensional algebras by Brenner and Butler \cite{BB} and Happel and Ringel
\cite{HR}. Subsequently, Miyashita \cite{MI} generalized this concept by introducing the notion of
$n$-tilting modules over arbitrary associative rings with identity.

There are several definitions of the notion of cotilting modules in the context of finitely generated
$R$-modules $C$ satisfying $\id(C_R)\leq 1$, which are generally called 1-cotilting modules; see
\cite{Ba3, C, CDT, CTT, CT, R}. As shown in \cite{H} and \cite{Ba3}, some definitions of 1-cotilting
modules coincide.

Unlike 1-cotilting modules, the definitions of cotilting modules with higher injective dimensions do
not fully align. Several definitions exist for such modules. In this paper, we focus on four of them.

Miyashita \cite{MI} gave a definition for cotilting modules over Noetherian rings. Miyachi \cite{M2}
extended this definition to arbitrary rings and obtained several interesting results within the framework
of coherent rings. Then Auslander and Reiten have given a definition for cotilting modules over Artin
algebras; see \cite{AR1} and \cite{AR2}. Angeleri H\"{u}gel
and Coelho \cite{HC} extended this notion to modules over general associative rings with
identity that are not necessarily finitely generated. These cotilting modules are now known as big cotilting
modules. Recently, Aihara, Araya, Iyama, Takahashi and Yoshiwaki \cite{AAITY} have proposed a new definition
of cotilting modules. Let $n$ be a nonnegative integer, and let $M$ be a right $R$-module with finite
injective dimension at most $n$. To distinguish between these four types of cotilting
modules, we employ the following terminology: If $M$ is cotilting
according to the definition of Miyashita, Auslander-Reiten, Angeleri H\"{u}gel-Coelho, or the definition
of Aihara-Araya-Iyama-Takahashi-Yoshiwaki, respectively, we call it an $n$-M-cotilting,
$n$-AR-cotilting, $n$-big cotilting, and $n$-AAITY-cotilting module. Since the precise
value of the injective dimension is not the primary focus of our discussion, we often omit
the prefix \enquote{$n$}. The literature on cotilting modules is extensive, with significant focus
dedicated to the study of big cotilting modules in particular; see e.g. \cite{AAITY, HPST, Ba1, Ba2,
Ba4, HT,  M1, M2, S1, S2, STH, Y}.

The main results of this paper are Theorems \ref{3.9a}, \ref{3.15a}, and \ref{3.18aaa}. Theorem \ref{3.9a}
compares M-cotilting and big cotilting modules, Theorem \ref{3.15a} compares big cotilting and AAITY-cotilting
modules, and Theorem \ref{3.18aaa} compares M-cotilting and AAITY-cotilting modules. As a consequence, over a
right Artinian Noetherian algebra, a finitely generated right module $C$ is M-cotilting if and only if it is big
cotilting and if and only if it is AAITY-cotilting; see Corollary \ref{3.20}. Moreover, over an Artin algebra,
these four notions of cotilting modules coincide for finitely generated right modules; see Corollary \ref{3.23}.

\section{Preliminaries}

Throughout, rings are associative with identity, and modules are assumed to be unitary. Let $R$ be an associative
ring with identity. We use the notation $M_R$ (respectively, $_RM$) to denote a right (respectively, left)
$R$-module. The category of all right $R$-modules (respectively, finitely generated right $R$-modules) is denoted
by $\text{Mod-}R$ (respectively, $\text{mod-}R$). Let $M_R$ be an $R$-module. We denote by $\Add(M)$ (respectively,
$\add(M)$) the class of right $R$-modules which are isomorphic to a direct summand of a direct sum of copies
(respectively, finitely many copies) of $M$. Dually, $\Prod(M)$ stands for the class of right $R$-modules which
are isomorphic to a direct summand of a direct product of copies of $M$. Also, we write $\text{Cogen}(M)$ for the
subcategory of all $R$-modules $N_R$ admitting an $R$-monomorphism $N\rt X$ with $X\in \Prod(M)$.

Let $\mathcal{C}$ be a subcategory of $\text{Mod-}R$. Let $\widehat{\mathcal{C}}$ denote the subcategory of modules
$N_R$ admitting an exact sequence $$0\rt C_n\rt \cdots \rt C_1\rt C_0\rt N\rt 0$$ with each $C_i$ in $\mathcal{C}$.
The notation $\mathcal{C}^{\perp}$ refers to the subcategory of $R$-modules $N_R$ such that
$\Ext^{i\geqslant1}_{R}(X,N)=0$ for every $X\in \mathcal{C}$. Dually, the notation $~^{\perp}\mathcal{C}$ stands
for the subcategory of $R$-modules $N_R$ such that $\Ext^{i\geqslant 1}_{R}(N,X)=0$ for every $X\in \mathcal{C}$.

For an $R$-module $T_R$, the symbol $\gen^*(T)$ stands for the class of $R$-modules $N_R$ for which there exists
an exact sequence of the form $$\cdots\st{f_2}\longrightarrow T_1 \st{f_1}\longrightarrow  T_0 \st{f_0}\longrightarrow
N\longrightarrow  0$$ with each $T_i\in \add(T)$ and $\Ext^1_R(T,\Ker \ f_i)=0$ for all $i\geqslant 0$. Dually, the
class $\Cogen^*(T)$ (respectively,  $\cogen^*(T)$) is consisting of $R$-modules $N_R$ for which there exists an exact
sequence of the form $$0\longrightarrow  N \st{f^{-1}}\longrightarrow  T^0\st{f^0}\longrightarrow  T^1\st{f^1}
\longrightarrow  \cdots$$ with each $T^i\in \Prod(T)$ (respectively,  $T^i\in \add(T)$) and $\Ext^1_R(\Coker \ f^i,T)=0$
for all $i\geqslant -1$.

The notions $\Add(-), \add(-), \Prod(-), \Cogen(-), \gen^*(-), \Cogen^*(-)$  and $\cogen^*(-)$ for left $R$-modules
are all defined analogously.

Recall that an $R$-module $M_R$ (respectively, $_RM$) is called {\it self-orthogonal} if $\Ext^{i\geq 1}_{R}(M,M)=0$.
Following \cite{MI} and \cite{W}, we present the following two definitions.

\begin{definition}\label{2.1}  Let $T_R$ be a self-orthogonal $R$-module and $n\in \mathbb{N}_0$. We say that $T_R$ is
$n$-{\it tilting} if
\begin{itemize}
\item[(i)]  $T_R\in \gen^*(R)$ and $\pd(T_R)\leq n$, and
\item[(ii)] There is an exact sequence $$0\rt R\rt T_0 \rt T_1\rt \cdots \rt T_n\rt 0,$$ where $T_i\in \add(T)$ for all
$0\leq i\leq n$.
\end{itemize}
\end{definition}

\begin{definition}\label{2.2} A self-orthogonal $R$-module $T_R$ is called a {\it Wakamatsu tilting} module if
\begin{itemize}
\item[(i)] $T_R\in \gen^*(R)$, and
\item[(ii)] $R_R \in \cogen^*(T)$.
\end{itemize}
\end{definition}

It is easy to verify that every $n$-tilting right $R$-module is also a Wakamatsu tilting right $R$-module. The notions of
$n$-tilting and Wakamatsu tilting left $R$-modules are defined similarly. By \cite[Corollary 3.2]{W}, we have the following
characterization of the Wakamatsu tilting modules.

\begin{lemma}\label{2.3} For any bimodule $_ST_R$, the following conditions are equivalent:

$(1)$   $T_R$ is a Wakamatsu tilting module with $\End(T_R)\cong S;$

$(2)$   $_ST$ is a Wakamatsu tilting module with $\End(_ST)\cong R;$

$(3)$ One has
\begin{itemize}
\item[(i)]  $T_R\in \gen^*(R)$  and $~_ST\in \gen^*(S)$.
\item[(i)] $\End(T_R)\cong S$ and $\End(_ST)\cong R$.
\item[(iii)]  The modules $T_R$ and $_ST$ are self-orthogonal.
\end{itemize}
\end{lemma}

From now on, when we say that $_S T_R$ is a Wakamatsu tilting bimodule, we mean that $T_R$ is a Wakamatsu tilting module
and $S \cong \End(T_R)$. If a bimodule $_ST_R$ is a Wakamatsu tilting module,  then by the above lemma, it follows that
both modules $T_R$ and $_ST$ are finitely presented.

\begin{lemma}\label{2.4} Let $_ST_R$ be a bimodule.
\begin{itemize}
\item[(i)] If $R$ is right Noetherian, then $\id((F\otimes_ST)_R)\leq \id(T_R)$ for every flat $S$-module $F_S$.
\item[(ii)] If $S$ is left coherent and $Q_R$ is an injective cogenerator for $\text{Mod-}R$, then $$\fd(\Hom_R(T,Q)_S)
=\id(_ST).$$
\end{itemize}
\end{lemma}

\begin{proof} (i) See \cite[Lemma 2.7(i)]{DMT}.

(ii) See \cite[Theorem 3.2.19 and Remarks 3.2.25 and 3.2.27]{EJ}.
\end{proof}

\section{Main Results}

In 1986, Miyashita presented the following definition for cotilting modules; see \cite[page 142]{MI}.

\begin{definition}\label{3.1} Let $C_R$ be an $R$-module and $S=\End(C_R)$. Assume that the ring $R$ is right
Noetherian and the ring $S$ is left Noetherian. We say that $C_R$ is {\it M-cotilting} if it is a Wakamatsu tilting
$R$-module and both $\id(C_R)$ and $\id(_SC)$ are finite.
\end{definition}

Let $A$ be a commutative Artinian ring with identity, and $R$ be an Artin $A$-algebra.  Let $\mathscr{E}$ be the minimal
injective cogenerator of $A$, and set $D(-)=\Hom_{A}(-,\mathscr{E})$. In 1991, Auslander and Reiten introduced the concept
of cotilting modules as an analogous to dualizing modules; see \cite{AR1} and \cite{AR2}.

\begin{definition}\label{3.2} Let $R$ and $D(-)$ be as above. A finitely generated self-orthogonal $R$-module
$C_R$ is called {\it AR-cotilting} if $\id(C_R)<\infty$, and there exists an exact sequence $$0\rt C_n\rt
\cdots \rt C_0\rt D(R)\rt 0$$ with each $C_i$ in $\add(C)$.
\end{definition}

It is easy to verify that an $R$-module $C_R$ is AR-cotilting if and only if $_RD(C)$ is an $n$-tilting $R$-module.
This definition was extended to (not necessarily finitely generated) modules over general associative rings with
identity by Angeleri H\"{u}gel and Coelho \cite{HC} in 2001.

\begin{definition}\label{3.3} An $R$-module $C_R$ is called \it{big cotilting} if
\begin{itemize}
\item[(i)] $\id(C_R)<\infty$,
\item[(ii)] $\Ext^{i\geq 1}_R(C^I,C)=0$ for every set $I$, and
\item[(iii)] there exists an injective cogenerator $Q_R$ for $\text{Mod-}R$ and an exact sequence $$0\rt C_n\rt \cdots
\rt C_0\rt Q\rt 0$$ with each $C_i$ in $\Prod(C)$.
\end{itemize}
\end{definition}

In 2014, Aihara, Araya, Iyama, Takahashi, and Yoshiwaki introduced the following definition for cotilting modules
over Noetherian rings; see \cite[Definition 5.1]{AAITY}. In this work, we extend their definition by relaxing the
Noetherian assumption.

\begin{definition}\label{3.4} A finitely generated self-orthogonal $R$-module $C_R$ is called \it{AAITY-cotilting} if
\begin{itemize}
\item[(i)] $\id(C_R)<\infty$, and
\item[(ii)] For any $X\in ~^\perp C\medcap \text{mod-}R$, there exists a short exact sequence $$0\rt X\rt C_0\rt X'
\rt 0$$ with $C_0\in \add(C)$ and $X'\in ~^\perp C$.
\end{itemize}
\end{definition}

The second condition in the above definition is equivalent to the inclusion $~^\perp C \medcap \text{mod-}R \subseteq
\cogen^*(C)$.

Whenever we consider any of the four types of cotilting modules discussed above, we assume that the rings satisfy the 
conditions required by their respective definitions.

The next result provides numerous examples of each type of the above cotilting modules.

\begin{example}\label{3.5}
\begin{itemize}
\item[(i)] Let $R$ be a connected finite-dimensional hereditary algebra over an algebraically closed field $\Bbbk$
of infinite representation type. Let $D(-):=\Hom_{\Bbbk}(-,\Bbbk)$. For any finitely generated $R$-module $N$, we set $\Tr~(N):=\Coker(\Hom_R(f,R)),$ where $P_1\st{f}\rt P_0\rt N\rt0$ is a minimal projective presentation
of $N$. Then the right $R$-module $(\Tr~ D)^n(R)$ is M-cotilting for every $n\geq 0$; see \cite[page 592]{M2}.
\item[(ii)] Let $\Gamma$ be a commutative Cohen-Macaulay local ring with a dualizing module $\omega$, and let $R$ be a
$\Gamma$-order (i.e. $R$ is a $\Gamma$-algebra and $R$ is maximal Cohen-Macaulay as a $\Gamma$-module).
By \cite[Proposition 2.12]{M2}, $\Hom_\Gamma(R,\omega)$ is an M-cotilting right $R$-module.
Furthermore, if $_RT$ is an $n$-tilting $R$-module that is also maximal Cohen-Macaulay as a $\Gamma$-module, then the
right $R$-module $\Hom_\Gamma(T,\omega)$ is AAITY-cotilting; see \cite[Example 5.2(1)]{AAITY}.
\item[(iii)] Let $R$ be an Artin algebra and $C_R$ an AR-cotilting $R$-module with $\id(C_R)=r$. Assume that $R\mathcal{Q}$
is the path algebra of the quiver $\mathcal{Q}:1\rightarrow 2 \rightarrow 3$. It is easy to see that $R\mathcal{Q}$ is
an Artin algebra. By \cite[Lemma 3.7]{Zh}, $$\textbf{C}=(0\rightarrow0 \rightarrow C)\oplus (0\rightarrow C \rightarrow C)
\oplus (C\rightarrow C\rightarrow C)$$ is an AR-cotilting right $R\mathcal{Q}$-module with $\id(\textbf{C}_{R\mathcal{Q}})=r+1$.
So, for any nonnegative integer $n$, we can construct an AR-cotilting module of injective dimension $n$.
\item[(iv)] Let $Q_R$ be an injective cogenerator for $\text{Mod-}R$. Then, clearly, $Q_R$ is a big cotilting module.
\item[(v)] Let $M=\mathbb{Q}\oplus\mathbb{Q/Z}$. Clearly, $M_{\mathbb{Z}}$ is an injective cogenerator of $\mathbb{Z}$. Let
$R=\left(\begin{matrix}
\mathbb{Z}& 0\\
\mathbb{Z}&\mathbb{Z}
\end{matrix}
\right).$
By \cite[Theorem 2.6] {Ma}, $C=(M\rt M\oplus M)$ is a big cotilting right $R$-module with $\id(C_R)=1$.
\end{itemize}
\end{example}

In Theorems \ref{3.9a}, \ref{3.15a}, and \ref{3.18aaa}, we compare the notions of M-cotilting, big cotilting, and
AAITY-cotilting modules. To prove these theorems, we require the following two lemmas.

\begin{lemma}\label{3.6} Let $C_R$ be a self-orthogonal $R$-module, and $S=\End(C_R)$. Assume that:\\
(i) The ring $S$ is left coherent.\\
(ii) There exists an injective cogenerator $Q_R$ for $\text{Mod-}R$ and an exact sequence $$\cdots \longrightarrow
C_n\overset{d_n}\longrightarrow \cdots \longrightarrow C_0\overset{d_0}\longrightarrow Q\longrightarrow 0$$ with
each $C_i\in \Prod(C)$ and $\Ext_R^1(C, \Ker \ d_i)=0$ for all $i\geq 0$.\\
Then, we have:
\begin{itemize}
\item[(a)]  If $_SC$ is finitely presented, then it is self-orthogonal and the natural homothety map $\chi:R\rt \Hom_S(C,C)$
is an isomorphism.
\item[(b)] If the ring $R$ is right coherent and both modules $C_R$ and $_SC$ are finitely presented, then $C_R$ is a Wakamatsu
tilting module.
\end{itemize}
\end{lemma}

\begin{proof} (a) Since $\Ext_R^1(C,\Ker \ d_i)=0$ for all $i\geq 0$, it follows that the sequence
\begin{equation}
\cdots \rt \Hom_R(C,C_n)\rt \cdots \rt \Hom_R(C,C_0)\rt \Hom_R(C,Q)\rt 0  \label{7acc}
\end{equation}
is exact. For each $i$, since $C_i\in \Prod(C)$ and $S$ is left coherent, it follows from \cite[Theorem 2.1]{Ch} that
$\Hom_R(C,C_i)$ is a flat right $S$-module. Thus, \eqref{7acc} is a flat resolution of the right $S$-module $\Hom_R(C,Q)$.

For any $R$-module $M_R$, let $$\nu_M:\Hom_R(C,M)\otimes_SC\rt M$$ denote the evaluation map. Given a set $I$, it is
straightforward to check that $\nu_{C^I}$ is the composition of the following natural isomorphisms
$$\Hom_R(C,C^I)\otimes_SC\st{\cong}\longrightarrow \Hom_R(C,C)^I\otimes_SC\st{\cong}\longrightarrow (\Hom_R(C,C)
\otimes_SC)^I\st{\cong}\longrightarrow (S\otimes_SC)^I\st{\cong}\longrightarrow C^I.$$ Note that, as $_SC$ is finitely
presented, the functor $-\otimes_SC$ commutes with arbitrary direct products, and so, in the above display, the second
natural map is an isomorphism. Hence, $\nu_{C^I}$ is an isomorphism.

Let $X_R\in \Prod(C)$. Then there are a set $I$ and an $R$-module $Y_R$ such that $X\oplus Y=C^I$. The
split short exact sequence $$0\rt X\st{\lambda}\longrightarrow C^I\st{\pi}\longrightarrow Y\rt 0$$ yields the following
commutative diagram with split exact rows:

\vspace{0.5cm}
\begin{tikzcd}[row sep=2em, column sep =1.5em]
0\arrow{r} &\Hom_R(C,X)\otimes_SC\arrow{r}\arrow{d}{\nu_{X}} &\Hom_R(C,C^I)\otimes_SC\arrow{r}\arrow{d}{\nu_{C^I}}	
&\Hom_R(C,Y)\otimes_SC\arrow{r}\arrow{d}{\nu_{Y}} &0\\
0\arrow{r} &X\arrow{r} &C^I\arrow{r} &Y\arrow{r} &0.
\end{tikzcd}
\vspace{0.3cm}\\
Since $\nu_{C^I}$ is an isomorphism, it follows that $\nu_{X}$ is injective and $\nu_{Y}$ is surjective. By interchanging
the roles of $X$ and $Y$ in the above argument, it follows that $\nu_{X}$ is also surjective. Hence,  $\nu_{X}$ is an
isomorphism.

From the exact sequence $$\cdots \longrightarrow C_n\overset{d_n}\longrightarrow \cdots \longrightarrow C_0\overset{d_0}
\longrightarrow Q\longrightarrow 0,$$ we can induce the following commutative diagram:

\vspace{0.5cm}
\begin{tikzcd}[row sep=2em, column sep =1.5em]
\cdots\arrow{r} &\Hom_R(C,C_1)\otimes_SC\arrow{r}\arrow{d}{\nu_{C_1}} \arrow{r}  &\Hom_R(C,C_0)\otimes_SC
\arrow{r}\arrow{d}{\nu_{C_0}}	&\Hom_R(C,Q)\otimes_SC\arrow{r}\arrow{d}{\nu_{Q}} &0\\
\cdots \arrow{r} &C_1\arrow{r} \arrow{r}  &C_0\arrow{r} &Q\arrow{r} &0.
\end{tikzcd}
\vspace{0.3cm}\\
As the two maps $\nu_{C_0}$ and $\nu_{C_1}$ are isomorphisms, applying the five lemma to the right side of the above
diagram implies that the map $\nu_{Q}$ is also an isomorphism. Since in the above commutative diagram the bottom row
is exact and each vertical map is an isomorphism, it follows that the top row is also exact. Thus, as \eqref{7acc}
is a flat resolution of the right $S$-module $\Hom_R(C,Q)$, we conclude that $$\Hom_R(\Ext_S^i(C,C),Q)\cong
\Tor^S_i(\Hom_R(C,Q),C)=0$$ for all $i\geq 1$. Since $Q$ is an injective cogenerator for $\text{Mod-}R$, it follows
that it is faithfully injective, and so $\Ext_S^i(C,C)=0$ for all $i\geq 1$.

Let $$\theta_Q: \Hom_R(C,Q)\otimes_SC\rt \Hom_R(\Hom_S(C,C),Q)$$ be the Hom-evaluation isomorphism. Also, let
$\mu: Q\rt \Hom_R(R,Q)$ be the natural isomorphism and $\alpha=\Hom_R(\chi,\text{id}_Q)$. We have the following
commutative diagram:

\[
\begin{tikzcd}[row sep=large,column sep=huge]
\Hom_R(\Hom_S(C,C),Q)  \arrow[dr, "\alpha"] &  \Hom_R(C,Q)\otimes_SC \arrow[l, "\theta_Q"]  \arrow[d, "\mu \nu_{Q}"]\\
& \Hom_R(R,Q).
\end{tikzcd}
\]
\vspace{0.3cm}\\
Thus the map $\alpha$ is an isomorphism, because the two maps $\theta_Q$ and $\mu \nu_{Q}$ are so. As $Q$ is
faithfully injective, it follows that $\chi$ is an isomorphism.

(b) Since $C_R$ is finitely presented and $R$ is right coherent, it follows that $C_R\in \gen^*(R)$. Similarly, as $_SC$
is finitely presented and $S$ is left coherent, we deduce that $_SC\in \gen^*(S)$. On the other hand, by the assumption
and (a), we have $\End(C_R)=S$, $\End(_SC)\cong R$, and the modules $C_R$ and $_SC$ are self-orthogonal. Thus, $C_R$ is
a Wakamatsu tilting module by Lemma \ref{2.3}.
\end{proof}

\begin{lemma}\label{3.7} Let $C_R$ be an M-cotilting $R$-module. Let $Q_R$ be an injective cogenerator for $\text{Mod-}R$.
\begin{itemize}
\item[(a)] There exists an exact sequence $0\rt C_m\rt \cdots \rt C_0\rt Q\rt 0$ in which each $C_i$ belongs to $\Add(C)$.
\item[(b)] Assume that $R$ is a Noetherian algebra and that $Q_R$ is finitely generated. Then there exists an exact sequence
$0\rt C_m\rt \cdots \rt C_0\rt Q\rt 0$ in which each $C_i$ belongs to $\add(C)$.
\end{itemize}
\end{lemma}

\begin{proof} (a) Set $S=\End(C_R)$, and let $m=\id(_SC)$.  We have $$\fd(\Hom_R(C,Q)_S)=\id(_SC)=m$$ by Lemma \ref{2.4}(ii). 
Thus, we have an exact sequence
\begin{equation}
0\rt F_{m+1}\rt S^{(I_m)}\rt S^{(I_{m-1})}\rt \cdots \rt S^{(I_0)}\rt \Hom_R(C,Q)\rt 0, \label{8}
\end{equation}
where $F_{m+1}$ is a flat right $S$-module. As the $R$-module $Q_R$ is injective, by the Hom-evaluation isomorphism, one
deduces that $$\Tor_{i\geq 1}^S(\Hom_R(C,Q),C)\cong \Hom_R(\Ext_S^{i\geq 1}(C,C),Q)=0$$ and $\Hom_R(C,Q)\otimes_SC\cong Q$.
Thus applying the functor $-\otimes_SC$ to \eqref{8} yields the following exact sequence
$$0\longrightarrow F_{m+1}\otimes_SC\longrightarrow C^{(I_{m})}\st{d_m}\longrightarrow C^{(I_{m-1})}\longrightarrow
\cdots \longrightarrow C^{(I_0)}\st{d_0}\longrightarrow Q\longrightarrow 0.$$

Set $K_{-1}=Q$ and $K_j=\Ker ~d_j$ for every $j=0,..., m$. Showing that $K_{m-1}\in \Add(C)$, would complete the proof.
We have the short exact sequences
\begin{equation}
0\rt K_j\rt C^{(I_j)}\rt K_{j-1}\rt 0\  \  \ j=0,..., m.  \label{9}
\end{equation}
By \cite[Theorem 3.2.15]{EJ}, we conclude that
\begin{equation}
\begin{array}{lllll}
\Ext_R^{i\geq 1}(C^{(I_j)},K_{m})& \cong \Ext_R^{i\geq 1}(C,K_{m})^{I_j}\\
& \cong \Ext_R^{i\geq 1}(C,F_{m+1}\otimes_SC)^{I_j}\\
& \cong (F_{m+1}\otimes_S\Ext_R^{i\geq 1}(C,C))^{I_j}\\
& =0.
\label{10}
\end{array}
\end{equation}

By Lemma \ref{2.4}(i) and \cite[Theorem 2.7]{Hu2}, we have $$\id((K_m)_R)=\id((F_{m+1}\otimes_SC)_R)\leq \id(C_R)=\id(_SC)=
m.$$ Hence, using the exact sequences \eqref{9} and the isomorphisms  \eqref{10} yield that $$\Ext_R^1(K_{m-1}, K_{m})\cong
\Ext_R^{m+1}(Q, K_{m})=0.$$ Thus the short exact sequence $$0\rt K_m\rt C^{(I_m)}\rt K_{m-1}\rt 0$$ splits, and so $K_{m-1}
\in \Add(C)$.

(b) By \cite[Proposition 4.2(c)]{A}, the ring $S$ is also a Noetherian algebra. Consequently, applying parts (b) and (a) of
\cite[Proposition 4.2]{A} implies that the right $S$-module $\Hom_R(C,Q)$ is finitely generated. Hence, as the ring $S$ is
right Noetherian, in \eqref{8} we can take the sets $I_0, \ldots, I_m$ to be finite. Therefore, the remainder of the above
argument shows that there exists an exact sequence $$0\rt C_m\rt \cdots \rt C_0\rt Q\rt 0$$ with each $C_i\in \add(C)$.
\end{proof}

Next, we recall the notion of product-complete modules.

\begin{definition}\label{3.8a} An $R$-module $C_R$ is called {\it product-complete} if $\Prod(C)\subseteq \Add(C)$.
\end{definition}

As noted in \cite[Remark 3.5]{MR}, an $R$-module $C_R$ is product-complete if and only if $\Prod(C)=\Add(C)$. Furthermore,
if $S=\End(C_R)$ is left Noetherian, then by \cite[Corollary 4.4]{KS}, the $R$-module $C_R$ is product-complete if and only
if $_S C$ has finite length.

The following presents our first main result, which provides a comparison between M-cotilting modules and big cotilting
modules.

\begin{theorem}\label{3.9a}
Let $C_R$ be a finitely generated $R$-module, and let $S=\End(C_R)$.
\begin{itemize}
\item[(a)] Assume that the ring $R$ is right Noetherian, the ring $S$ is left Noetherian, and $_SC$ is finitely generated.
If $C_R$ is a big cotilting module, then it is also M-cotilting.
\item[(b)] Assume that $C_R$ is product-complete.
If $C_R$ is an M-cotilting module, then it is also big cotilting.
\end{itemize}
\end{theorem}

\begin{proof}
(a) Since $C_R$ is a big cotilting module, there exists an injective cogenerator $Q_R$ for $\text{Mod-}R$ and an exact
sequence

$$0\longrightarrow C_n\st{d_n}\longrightarrow  \cdots \longrightarrow  C_0\st{d_0}\longrightarrow  Q\longrightarrow 0$$
with each $C_i$ in $\Prod(C)$.

For each $j>0$, since the functor $\Ext_R^j(C,-)$ commutes with direct products, we have $\Ext_R^j(C,C_i)=0$ for all
$0\leq i\leq n$. Set $X_i = \im(d_i)$ for each $1\leq i\leq n$. Then we obtain the short exact sequences
$$0\rt C_n \rt C_{n-1}\rt X_{n-1}\rt 0,$$
$$0\rt X_i \rt C_{i-1}\rt X_{i-1}\rt 0, \  \ i=n-1, n-2, \dots, 2$$ and
$$0\rt X_1 \rt C_0\rt Q\rt 0.$$
From these sequences, it follows that $\Ext_R^1(C,X_i)=0$ for all $i=n-1,n-2,\dots,1$, and hence the sequence
\begin{equation}
0 \rt \Hom_R(C,C_n) \rt \cdots \rt \Hom_R(C,C_0) \rt \Hom_R(C,Q) \rt 0 \label{7a}
\end{equation}
is exact. Since $\Ext_R^1(C,\Ker \ d_i)=0$ for all $i\geq 0$, Lemma \ref{3.6}(b) implies that $C_R$ is Wakamatsu
tilting.

As noted in the proof of Lemma \ref{3.6}(a), each $\Hom_R(C,C_i)$ is a flat right $S$-module.
Thus, by Lemma \ref{2.4}(ii), we obtain $$\id(_SC)=\fd(\Hom_R(C,Q)_S)\leq n.$$
Therefore, $C_R$ is M-cotilting.

(b) Since $C_R$ is product-complete, we have $\Add(C)=\Prod(C)$. Thus, for any set $I$, there is a set $J$ such that $T^I$
is isomorphic to a direct summand of $T^{(J)}$. For any $i>0$, since $$\Ext^i_R(T^{(J)},T)\cong \Ext^i_R(T,T)^J=0,$$ it
follows that $\Ext^i_R(T^I,T)=0$. Therefore, the claim follows from Lemma \ref{3.7}(a).
\end{proof}

The following example shows that the assumption product-complete can't be relaxed in Theorem \ref{3.9a}(b).

\begin{example}\label{3.10a} Clearly, $\mathbb{Z}_{\mathbb{Z}}$ is an M-cotilting $\mathbb{Z}$-module. By
\cite[Example 2.1]{H}, we know that $\Ext_{\mathbb{Z}}^1(\mathbb{Z}^{\mathbb{N}},\mathbb{Z})\neq 0$, and
so $\mathbb{Z}_{\mathbb{Z}}$ is not a big cotilting $\mathbb{Z}$-module.
\end{example}

In the proof of Lemma \ref{3.12aaa}, we will use the following lemma.

\begin{lemma}\label{3.11a} Let $C_R$ be an $R$-module. Assume that $\Ext^{i\geq 1}_R(C^I,C)=0$ for every set $I$,
and $~^\perp C\subseteq \text{Cogen}(C)$. Then for every $X\in ~^\perp C$, there exists a short exact sequence
$$0\rt X\rt C_0\rt X'\rt 0$$ with $C_0\in \Prod(C)$ and $X'\in ~^\perp C$.
\end{lemma}

\begin{proof}  See \cite[Lemma 2.4(b)]{HC}.
\end{proof}

The next result establishes an equivalent condition to the third criterion in the definition of big cotilting modules.

\begin{lemma}\label{3.12aaa} Assume that an $R$-module $C_R$ satisfies the following conditions:\\
(i) $\id(C_R)<\infty$.\\
(ii) $\Ext^{i\geq 1}_R(C^I,C)=0$ for every set $I$. \\
Then the following are equivalent:
\begin{itemize}
\item[(a)] There exists an injective cogenerator $Q_R$ for $\text{Mod-}R$ and an exact sequence
$$0\rt C_n\rt \cdots \rt C_0\rt Q\rt 0$$
with each $C_i\in \Prod(C)$.
\item[(b)] $~^\perp C\subseteq \Cogen(C)$.
\item[(c)] For any $X\in ~^\perp C$, there exists a short exact sequence
$$0\rt X\rt C_0\rt X'\rt 0$$
with $C_0\in \Prod(C)$ and $X'\in ~^\perp C$. In other words, $~^\perp C\subseteq \Cogen^*(C)$.
\end{itemize}
\end{lemma}

\begin{proof} (a)$\Longleftrightarrow$(b) follows from the last paragraph of \cite{HC}.

(b)$\Longrightarrow$(c) follows from Lemma \ref{3.11a}.

(c)$\Longrightarrow$(b) is immediate.
\end{proof}

The following two lemmas will be used in the proof of Theorem \ref{3.15a}.

\begin{lemma}\label{3.13aaa} Let $C_R$ be a finitely generated self-orthogonal $R$-module. Assume that $C_R$ is
product-complete. Then $$~^{\perp} C\medcap \Cogen^*(C)\medcap \text{mod-}R\subseteq ~^{\perp} C\medcap
\cogen^*(C).$$ Moreover, if $R$ is  right Noetherian, then equality holds.
\end{lemma}

\begin{proof} Let $X_R$ be an $R$-module in $~^{\perp}C\medcap \Cogen^*(C)\medcap \text{mod-}R$. Since $X\in
~^{\perp} C\medcap \Cogen^*(C)$ and $\Prod(C)=\Add(C)$, we can construct a short exact sequence
\begin{equation}
0\rt X\rt C_0\rt X'\rt 0 \label{11a}
\end{equation}
with $C_0\in \Add(C)$ and $X'\in ~^\perp C$. As $C_0\in \Add(C)$, there exists a set $I$ and an $R$-module
$K_R$ such that $C_0\oplus K\cong C^{(I)}$. Using \eqref{11a}, we obtain a short exact sequence
\begin{equation}
0\rt X\overset{f}\rt C^{(I)}\rt X''\rt 0, \label{12a}
\end{equation}
where $X''=X'\oplus K$.

Since $X$ is finitely generated, there exists a finite subset $J$ of $I$ such that $\pi'(\im\, f)=0$, where
$\pi':C^{(I)}\rt C^{(I\setminus J)}$ is the natural projection. Let $\pi:C^{(I)}\rt C^{(J)}$ be the natural
projection and set $g=\pi f$. Thus we obtain the short exact sequence
\begin{equation}
0\rt X\overset{g}\rt C^{(J)}\rt \Coker\, g\rt 0. \label{13a}
\end{equation}
Obviously, $C^{(J)}\in \add(C)$, and we show that $\Coker, g\in ~^\perp C$. Since $X''\in ~^\perp C$, applying
the functor $\Hom_R(-,C)$ to \eqref{12a} yields that the map $\Hom_R(f,C)$ is surjective. Let $$\lambda: C^{(J)}
\rt C^{(I)}$$ be the natural inclusion map. Then $f=\lambda g$, and so $$\Hom_R(f,C)=\Hom_R(g,C)\Hom_R(\lambda,C).$$
Since $\Hom_R(f,C)$ is surjective, it follows that $\Hom_R(g,C)$ is also surjective. Using this, the long exact
sequence of Ext groups obtained by applying $\Hom_R(-,C)$ to \eqref{13a} shows that $\Ext_R^i(\Coker\, g,C)=0$
for all $i\geq 1$, and hence $\Coker\, g\in ~^\perp C$. Continuing in this manner yields that $$X\in ~^{\perp}
C\medcap \cogen^*(C).$$

The final assertion is clear.
\end{proof}

\begin{lemma}\label{3.14aaa} Let $\mathscr{X}$ be a class of right $R$-modules that is closed under extensions, and let
$\mathscr{W}\subseteq \mathscr{X}$. Suppose that for every $X\in \mathscr{X}$, there exists a short exact sequence
$0\rt X\rt C\rt X'\rt 0$ with $C\in \mathscr{W}$ and $X'\in \mathscr{X}$. Then for every $X\in \widehat{\mathscr{X}}$,
there exists a short exact sequences $0\rt K_X\rt C_X\rt X\rt 0$ with $C_X\in \mathscr{X}$ and $K_X\in \widehat{\mathscr{W}}$.
\end{lemma}

\begin{proof}  See \cite[Theorem 3.2(i)]{HC}.
\end{proof}

We are now prepared to present our second main result, which provides a comparison between big cotilting modules and
AAITY-cotilting modules.

\begin{theorem}\label{3.15a} Let $C_R$ be a finitely generated product-complete $R$-module.
\begin{itemize}
\item[(a)] If $C_R$ is a big cotilting module, then it is also AAITY-cotilting.
\item[(b)] Assume that the ring $R$ is right Noetherian and $\text{Mod-}R$ has a finitely generated injective cogenerator.
If $C_R$ is an AAITY-cotilting module, then it is also big cotilting.
\end{itemize}
\end{theorem}

\begin{proof} (a) As $C_R$ is a big cotilting module, Lemma \ref{3.12aaa} implies that $~^\perp C\subseteq \Cogen^*(C)$. 
Thus, by Lemma \ref{3.13aaa}, we deduce that $~^{\perp} C\medcap \text{mod-}R\subseteq \cogen^*(C).$ This yields that 
$C_R$ is an AAITY-cotilting $R$-module.

(b) Since $\Prod(C)=\Add(C)$ and $\Ext^{i\geq 1}_{R}(C,C)=0$, it follows that $\Ext^{i\geq 1}_R(C^I,C)=0$ for every
set $I$.

Let $Q_R$ be a finitely generated injective cogenerator for $\text{Mod-}R$. Set $\mathscr{H}=~^\perp C$, $\mathscr{X}
=~^{\perp}C\medcap \text{mod-}R$ and $\mathscr{W}=\add(C)$. By adopting the proof of \cite[Lemma 2.2(b)]{HC}, we can
show that $\widehat{\mathscr{X}}=\text{mod-}R$, and so $Q\in \widehat{\mathscr{X}}$. Since $C_R$ is an AAITY-cotilting,
the assumptions of Lemma \ref{3.14aaa} are satisfied for the pair $\mathscr{W}\subseteq \mathscr{X}.$ Hence,
there is a short exact sequence
\begin{equation}
0\rt K_Q\rt C_Q\rt Q\rt 0, \label{14}
\end{equation}
with $C_Q\in \mathscr{X}$ and $K_Q\in \widehat{\mathscr{W}}$. As $K_Q\in \widehat{\mathscr{W}}$, there is an exact
sequence
\begin{equation}
0\rt C_n\rt \cdots \rt C_1\rt K_Q\rt 0 \label{15}
\end{equation}
with each $C_j\in \mathscr{W}$, and so $K_Q\in \mathscr{H}^\perp$. Since also $Q$ belongs to $\mathscr{H}^\perp$, by
\eqref{14}, it follows that $C_Q\in \mathscr{H}^\perp$. As $C_R$ is an AAITY-cotilting and $C_Q\in \mathscr{X}$, there
is a short exact sequence
\begin{equation}
0\rt C_Q\rt C_0\rt X'\rt 0, \label{16}
\end{equation}
with $C_0\in \add(C)$ and $X'\in \mathscr{H}$. Since $X'\in \mathscr{H}$ and $C_Q\in \mathscr{H}^\perp$, we deduce
that $\Ext_R^1(X',C_Q)=0$, so the sequence \eqref{16} splits and $C_Q$ lies in $\add(C)$. Now, combining the exact
sequences \eqref{14}
and \eqref{15} yields the exact sequence $$0\rt C_n\rt \cdots \rt C_1\rt C_0\rt Q\rt 0,$$ where $C_0=C_Q$. Since
$\add(C)\subseteq \Prod(C)$, this concludes the proof of the claim.
\end{proof}

For the following two definitions, we refer the reader to \cite{AR2} and \cite{Hu1}.

\begin{definition}\label{3.16aa} Let $_ST_R$ be a Wakamatsu tilting bimodule. Assume that the ring $R$ is right
Noetherian and the ring $S$ is left Noetherian. A finitely generated $R$-module $M_R$ is said to have {\it generalized
Gorenstein dimension zero} with respect to $T$ if the following conditions are satisfied:
\begin{itemize}
\item[(i)] $\Ext^{i\geqslant1}_R(M,T)=0$.
\item[(ii)] $\Ext^{i\geqslant1}_S(\Hom_R(M,T),T)=0$.
\item[(iii)] The natural map $\theta_M: M\to \Hom_S(\Hom_R(M,T),T)$ is an isomorphism.
\end{itemize}
We use $\mathcal{G}_T$ to denote the full subcategory of $\text{mod-}R$ consisting of the modules with generalized
Gorenstein dimension zero with respect to $T$.
\end{definition}

\begin{definition}\label{3.17aa} Let $_ST_R$ be a Wakamatsu tilting bimodule. Assume that the ring $R$ is right
Noetherian and the ring $S$ is left Noetherian. A finitely generated $R$-module $M_R$ is said to have {\it generalized
Gorenstein dimension at most $n$} with respect to $T$, denoted by $\text{G-dim}_{T}(M_R)\leq n$, if there exists an exact
sequence $$0\rt G_n\rt \cdots \rt G_1\rt G_0\rt M\rt 0$$ with each $G_i\in \mathcal{G}_T$.
\end{definition}

Finally, we are in a position to state our concluding main result, which provides a comparison between M-cotilting
modules and AAITY-cotilting modules.

\begin{theorem}\label{3.18aaa} Let $C_R$ be an $R$-module, and let $S=\End(C_R)$.
\begin{itemize}
\item[(a)] If $C_R$ is an M-cotilting module, then it is also AAITY-cotilting.
\item[(b)] Assume that the ring $R$ is right Noetherian and the ring $S$ is left Noetherian. If $C_R$ is an
AAITY-cotilting module, then it also is M-cotilting.
\end{itemize}
\end{theorem}

\begin{proof} (a) Since $C_R$ is M-cotilting, by \cite[Proposition 5.6]{W}, we have $^{\perp}C\medcap
\text{mod-}R\subseteq \cogen^*(C)$. Hence, $C_R$ is AAITY-cotilting.

(b) Let $n=\id(C_R)$. Since $C_R$ is finitely generated and $R$ is right Noetherian, we have $C_R \in \gen^*(R)$.
Moreover, as $R\in ~^\perp C\medcap \text{mod-}R$, the second condition in the definition of AAITY-cotilting modules
implies $R_R\in \cogen^*(C)$. Hence, $C_R$ is a Wakamatsu tilting $R$-module. It remains to show that $\id(_SC)<\infty$.

Because $C_R$ is AAITY-cotilting, $^\perp C \medcap \text{mod-}R\subseteq \cogen^*(C)$. Thus, by \cite[Proposition
2.5]{Hu1}, $$\CY_{C}=~^\perp{C}\medcap \cogen^*(C)=(^\perp{C}\medcap \text{mod-}R)\medcap \cogen^*(C)=~^\perp C \medcap
\text{mod-}R.$$ By \cite[Proposition 2.3]{Hu1}, this equality implies that for every $R$-module $M_R$ in $^\perp C
\medcap \text{mod-}R$, the natural map $$\theta_M: M\to \Hom_S(\Hom_R(M,C),C)$$ is injective. Combining this with
$\id(C_R)=n$ and applying \cite[Theorem 3.8]{Hu1}, we conclude that $\text{G-dim}_{C}(M_R)\leq n$ for every
finitely generated $R$-module $M_R$. Finally, \cite[Theorem 3.6]{ZY} yields $\id(_SC)=n$.
\end{proof}

We conclude the paper with three corollaries of Theorems \ref{3.9a}, \ref{3.15a} and \ref{3.18aaa}.

\begin{corollary}\label{3.19} Let $C_R$ be a finitely generated product-complete $R$-module with $S=\End(C_R)$.
Assume that the ring $R$ is right Noetherian and the ring $S$ is left Noetherian. Then the following conditions are
equivalent:
\begin{itemize}
\item[(a)] $C_R$ is an M-cotilting module.
\item[(b)] $C_R$ is a big cotilting module.
\item[(c)] $C_R$ is an AAITY-cotilting module.
\end{itemize}
\end{corollary}

\begin{proof} As the ring $S$ is left Noetherian and $C_R$ is product-complete, \cite[Corollary 4.4]{KS}
implies that $_SC$ has finite length. In particular, the $S$-module $_SC$ is finitely generated. Now, the
equivalence (a) and (b) is immediate by Theorem \ref{3.9a}, while the equivalence of (a) and (c) follows
from Theorem \ref{3.18aaa}.
\end{proof}

Keeping Corollary \ref{3.19} in mind, we investigate when finitely generated right $R$-modules are
product-complete. The following result shows that this holds over right Artinian Noetherian algebras.

\begin{corollary}\label{3.20}
Let $R$ be a Noetherian algebra and let $C_R$ be a finitely generated $R$-module. Assume that $R$ is right Artinian.
Then $C_R$ is product-complete, and the following are equivalent:
\begin{itemize}
\item[(a)] $C_R$ is an M-cotilting module.
\item[(b)] $C_R$ is a big cotilting module.
\item[(c)] $C_R$ is an AAITY-cotilting module.
\end{itemize}
\end{corollary}

\begin{proof} Set $S=\End(C_R)$. There exists a commutative Noetherian ring $A$ with identity and a ring homomorphism
$f\colon A\to R$ such that $R$ is finitely generated as an $A$-module via $f$. By \cite[Proposition 4.2(c)]{A}, the 
ring $S$ is also a Noetherian algebra. In particular, $S$ is left Noetherian. Therefore, by Corollary \ref{3.19}, to 
complete the proof it suffices to show that $C_R$ is product-complete.

Since $C_R$ is finitely generated and $R$ is right Artinian, $C_R$ has finite length. By \cite[Proposition 4.2(d)]{A}, 
the module $C$ has finite length as an $A$-module, and applying \cite[Proposition 4.2(d)]{A} once more shows that $_SC$ 
has finite length as well. As $S$ is left Noetherian and $_SC$ has finite length, \cite[Corollary 4.4]{KS} implies that 
$C_R$ is product-complete.
\end{proof}

To present the second corollary, we first recall the notion of right Morita rings; see e.g. \cite{H} and \cite{He}.

\begin{definition}\label{3.21}
A ring $R$ is called {\it right Morita} if there exist a ring $S$ and a bimodule $_S U_R$ such that the functors
$\Hom_R(-,U)\colon \text{mod-}R \to S\text{-mod}$ and $\Hom_S(-,U)\colon S\text{-mod}\to \text{mod-}R$ induce a
duality of categories.
\end{definition}

It is straightforward to verify that every Artin algebra is a right Morita ring. Conversely, right Morita rings
are not necessarily Artin algebras; see \cite[Example 3.6]{H}.

Let $R$ be a right Morita ring with a bimodule $_S U_R$ as in the above definition. By \cite[Corollary 24.9]{AF},
$R$ is right Artinian, and $U_R$ is a finitely generated injective cogenerator for $\text{Mod-}R$.

\begin{corollary}\label{3.22}
Let $R$ be a Noetherian algebra and let $C_R$ be a finitely generated $R$-module. Assume that $R$ is right Morita with
a finitely generated injective cogenerator $U_R$. Then the following conditions are equivalent:
\begin{itemize}
\item[(a)] $C_R$ is an M-cotilting module.
\item[(b)] $C_R$ is self-orthogonal, $\id(C_R)<\infty$, and there exists an exact sequence
$$0\rt C_n\rt \cdots \rt C_0\rt U\rt 0$$
with each $C_i$ in $\add(C)$.
\item[(c)] $C_R$ is a big cotilting module.
\item[(d)] $C_R$ is an AAITY-cotilting module.
\end{itemize}
\end{corollary}

\begin{proof}
Since $R$ is right Morita, it is right Artinian. Hence, by Corollary \ref{3.20}, $C_R$ is product-complete, and
conditions (a), (c), and (d) are equivalent.

\smallskip
(a)$\Longrightarrow$(b) Since $R$ is a Noetherian algebra, the assertion follows from Lemma \ref{3.7}(b).

\smallskip
(b)$\Longrightarrow$(c) As $C_R$ is self-orthogonal and product-complete, we have
$\Ext^{i\ge 1}_R(C^I,C)=0$ for every set $I$. Since $\add(C)\subseteq \Prod(C)$, the claim follows.
\end{proof}

Example \ref{3.10a} shows that the assertion of the above result does not hold for arbitrary Noetherian
algebras.

\begin{corollary}\label{3.23}
Let $R$ be an Artin algebra, and let $C_R$ be a finitely generated $R$-module. The following
conditions are equivalent:
\begin{itemize}
\item[(a)] $C_R$ is an M-cotilting module.
\item[(b)] $C_R$ is an AR-cotilting module.
\item[(c)] $C_R$ is a big cotilting module.
\item[(d)] $C_R$ is an AAITY-cotilting module.
\end{itemize}
\end{corollary}

\begin{proof} There exists a commutative Artinian ring $A$ with identity such that $R$ is a Noetherian $A$-algebra.
Let $\mathscr{E}$ be the minimal injective cogenerator of $A$, and set $U=\Hom_{A}(R,\mathscr{E})$. Then $U_R$ is a 
finitely generated injective cogenerator for $\text{Mod-}R$. The claim now follows immediately from Corollary \ref{3.22}.
\end{proof}


\end{document}